\documentclass[12pt]{amsart}
\usepackage{amsmath,amssymb}
\usepackage[colorlinks=true,linkcolor=black,citecolor=black]{hyperref}
\usepackage{geometry}

\newcommand{\mc}{\mathcal}

\newcommand{\leqs}{\leqslant }
\newcommand{\geqs}{\geqslant }

\newtheorem{lemma}{Lemma}
\newtheorem{proposition}{Proposition}
\newtheorem{theorem}{Theorem}
\newtheorem{cor}{Corollary}

\theoremstyle{definition}

\begin{document}

\author{Winston Heap}
\address{Max Planck Institute for Mathematics, Vivatsgasse 7, 53111 Bonn.}
\email{winstonheap@gmail.com}

\author{K. Soundararajan}
\address{Department of Mathematics, Stanford University, Stanford, CA 94305, USA.}
\email{ksound@stanford.edu}

\title{Lower bounds for moments of zeta and $L$-functions revisited}
\maketitle

\begin{abstract} 
This paper describes a method to compute lower bounds for moments of $\zeta$ and $L$-functions.  
The method is illustrated in the case of moments of $|\zeta(\tfrac 12+it)|$, where the results are new 
for small moments $0 <k <1$.   
\end{abstract}

\section{Introduction}

\noindent This paper reexamines the problem of obtaining lower bounds of the correct order of magnitude for moments of the Riemann zeta function on the critical line, and related problems for central values in families of $L$-functions.  Our work is motivated by recent work on the complementary problem of obtaining upper bounds for such moments.  For example, \cite{RS} enunciates the principle that an upper bound for a particular moment (with a little flexibility) may be used to establish upper bounds of the correct order of magnitude for all smaller moments.  Recent work of the authors with Radziwi{\l \l} \cite{HRS} provides such upper bounds for all moments of the Riemann zeta-function below the fourth moment.   In those papers, one key idea is to approximate Euler products that mimic suitable powers of the zeta-function using Dirichlet series of small length.  The aim of this paper is to demonstrate how that idea may also be used to establish lower bounds of the right order of magnitude for all moments of the Riemann zeta-function.  

\begin{theorem}  \label{thm1}  Let $T$ be large.    Uniformly for $(\log T)^{-\frac 12} \le k \le (\log T)^{\frac 12 -\delta}$ (for any fixed $\delta  >0$) we have  
$$ 
\int_T^{2T} |\zeta(\tfrac 12+ it)|^{2k} dt \ge C_k T (\log T)^{k^2}, 
$$ 
where we may take $C_k = C_1 k$ in the range $k\le 1$, and $C_k =  (C_2 k^2 \log (ek))^{-k^2}$ for some absolute positive constants $C_1$ and $C_2$.  
\end{theorem}   
 
    There is a long history concerning such lower bounds for $\zeta$ 
and $L$-functions.   To place our result briefly in context, we recall that in the range $k \ge 1$ such a lower bound was established by \cite{RS lower}, although our quantification of $C_k$ is better and the proof arguably simpler.   Theorem 1 is  new in the range $0< k\le 1$.  Previous work of Heath-Brown \cite{HB} had established such a bound for rational $k$ in this range, and for real $k$ such a bound was known to hold conditional on the Riemann Hypothesis (see \cite{HB, Ram1, Ram2}).   In the range 
$c(\log \log T)^{-\frac 12} \le k = o(1)$, Laurinchikas \cite{L} has shown that the $2k$-th moment is $\sim T (\log T)^{k^2}$.  
The constant  $C_k$ in our result tends to zero as $k\to 0$; with more effort, our argument could be made to yield $C_k \gg 1$ for all $k\le 1$, but we have not done so in the interest of  keeping the exposition simple.  

 Combining the upper bound result of Heap, Radziwi\l\l, and Soundararajan \cite{HRS} with the lower bound
 of Theorem 1,  we obtain  the following corollary.  
 
\begin{cor} For $T$ large, uniformly for $(\log T)^{-\frac 12} \le  k \le 2$ we have 
\[
T(\log T)^{k^2} \gg \int_T^{2T}|\zeta(\tfrac{1}{2}+it)|^{2k}dt \gg k T(\log T)^{k^2}.
\]
\end{cor}

The moments of $\zeta(\tfrac 12+it)$ encode information on the  distribution of large values of $|\zeta(\tfrac 12+it)|$.  
In \cite{S} it was observed that the $2k$-th moment of $|\zeta(\tfrac 12+it)|$ should  be dominated by values of  size $(\log  T)^{k}$, 
which should occur on  a set of measure about $T/(\log T)^{k^2}$.  On RH, it was shown in \cite{S}  that the measure of 
$\{ t \in [T,2T] : |\zeta (\tfrac 12+it)|\ge (\log  T)^k \} $ is $T(\log T)^{-k^2+o(1)}$ for any fixed positive $k$.  From Corollary 1, we may obtain a sharper form of such a result unconditionally in the limited range $0 <k<2$. 

\begin{cor}  Uniformly in the range 
$$
\sqrt{\log \log T} \log \log \log T \le  V \le 2 \log \log T - 2 \sqrt{\log \log T} \log \log \log T
$$
 we have 
$$ 
\text{meas}\{ t\in [T, 2T]: \ |\zeta(\tfrac 12+it)| \ge e^V \}  = T \exp\Big( -\frac{V^2}{\log  \log T} + O\Big( \frac{V \log \log  \log  T}{\sqrt{\log \log T}} \Big)\Big). 
$$
\end{cor}  

Recall that Selberg's central limit theorem (see \cite{RS2} for a proof related to ideas of this paper) states that for $t$ chosen uniformly from $[T,2T]$, $\log |\zeta(\tfrac 12+it)|$ has an approximately normal distribution with mean $0$ and variance $\sim \frac 12\log \log T$.   Radziwi{\l \l} \cite{R} has established a uniform version of this result showing that for $V\le (\log \log T)^{\frac 35-\epsilon}$   one has 
$$ 
\text{meas } \{ t\in [T, 2T]: \ \log |\zeta(\tfrac 12+it)| \ge V = \Delta \sqrt{\tfrac 12\log \log T}\}  
\sim  \frac{T}{\sqrt{2\pi}} \int_{\Delta}^{\infty} e^{-x^2/2} dx. 
$$ 
Corollary 2 gives a crude version of such a result but in a wider range for $V$.  
 
 \medskip
 
 {\bf Acknowledgments.}  The second author is partially supported by grants from the NSF (including the FRG grant DMS1854398), and through a Simons Investigator grant from the Simons Foundation.   We are grateful to Maksym Radziwi{\l \l} for many valuable discussions on these themes.

\section{Setup and plan of the proof}
\label{setup sec}

\noindent Since Theorem 1 is really new only in the range $0< k\le 1$, we give a detailed proof in this range.   In Section 6, we briefly indicate the modifications to the argument needed to establish Theorem 1 for $k\ge 1$, and also discuss lower bounds for moments of central values of $L$-functions in families.

Throughout, $\log_j$ will denote the $j$-fold iterated logarithm.   Let $T$ be large and assume that $1/\sqrt{\log T} \le k \le 1$.  
Let $\ell$ denote the largest integer such that $\log_{\ell} T \geq 10^4$.  Define a sequence $T_j$ by setting 
$T_1 = e^2$, and for $2\leqs j \leqs \ell$ by 
$$ 
 T_{j} := \exp \Big ( \frac{k\log T}{(\log_{j} T)^2} \Big ).  
$$
Note that $T_2$ is already large.  Further, the sequence $T_j$ is in ascending order, and lastly $k \log T \ll \log T_\ell \le 10^{-8} k \log T$.

For each $2 \leqs j \leqs \ell$, set
$$
\mathcal{P}_{j}(s) := \sum_{T_{j - 1} \leqs p < T_j} \frac{1}{p^s}, \qquad \text{ and } \qquad P_j = {\mathcal P}_j(1) = \sum_{T_{j-1} \leqs p < T_j} \frac 1p.  
$$
Note that
$$
P_j = \log \frac{\log T_{j}}{\log T_{j-1}} +O\Big( \frac{1}{\log T_{j-1}}\Big)  \sim 2 \log \Big(\frac{\log_{j-1} T}{\log_j T}\Big) = 2\log_j T - 2\log_{j+1} T,
$$ 
so that $P_\ell \geqs 10^4$, $P_{\ell-1} \geqs \exp(10^4)$, and so on. 

Let ${\mathcal N}$ denote the set of integers $n = n_2 \cdots n_\ell$ where each $n_j$ is divisible only by primes in the 
interval $T_{j-1}$ to $T_j$ and such that $\Omega(n_j) \le K_j:=500P_j$ for all $2\le j\le \ell$.  If $n\in {\mathcal N}$ then 
\begin{equation} 
\label{1} 
n = n_2 \cdots n_\ell \le T_2^{500P_2} T_3^{500P_3} \cdots T_\ell^{500 P_\ell} \le T^{k/9}.
\end{equation} 
Let $g(n)$ denote the multiplicative function given on prime powers by $g(p^{r}) = 1/r!$.   
Define, for any real number $\alpha$ and $2\le j \le \ell$ 
\begin{equation} 
\label{1.1} 
{\mathcal N}_j(s,\alpha) = \sum_{r=0}^{K_j} \frac{1}{r!} (\alpha {\mathcal P}_j(s,\alpha))^r = \sum_{\substack{ p|n \implies T_{j-1} \le p\le T_j \\ \Omega(n) \le K_j}} \frac{\alpha^{\Omega(n)} g(n)}{n^s},
\end{equation} 
and put  
\begin{equation} \label{2}
\mathcal{N}(s, \alpha) := \sum_{n\in {\mathcal N}} \frac{\alpha^{\Omega(n)} g(n)}{n^s} = \prod_{j=2}^{\ell} {\mathcal N}_j(s,\alpha). 
\end{equation} 
In view of \eqref{1}, ${\mathcal N}(s,\alpha)$ is a short Dirichlet polynomial.  
The idea is that ${\mathcal N}(s,\alpha)$ behaves in many ways like $\zeta(s)^{\alpha}$, but with the advantage that since ${\mathcal N}(s,\alpha)$ is a short Dirichlet polynomial, one can compute mean-values involving it and $\zeta(s)$.  
The proof of our theorem rests on the following three propositions dealing with such mean values involving $\zeta(s)$ and ${\mathcal N}(s,\alpha)$ for suitable values of $\alpha$. 

\begin{proposition}  \label{prop1} Let $T$ be large.  Uniformly in the range $1\ge k \ge 1/\sqrt{\log T}$ we have  
$$ 
\int_T^{2T} \zeta(\tfrac 12+it) {\mathcal N}(\tfrac 12+it, k-1) {\mathcal N}(\tfrac 12-it, k) dt \ge C_1 T (\log T)^{k^2},  
$$ 
for some positive constant $C_1$.
\end{proposition}

\begin{proposition}  \label{prop2} Let $T$ be large.  Uniformly in the range $1\ge k \ge 1/\sqrt{\log T}$ we have  
$$ 
\int_T^{2T} |\zeta(\tfrac 12+it) {\mathcal N}(\tfrac 12+it,k-1)|^2 dt \le C_2 k^{-1} T(\log T)^{k^2}, 
$$ 
for some positive constant $C_2$.  
\end{proposition} 

\begin{proposition} 
\label{prop3}  Let $T$ be large.   Uniformly in the range $1\ge k \ge 1/\sqrt{\log T}$ we have  
$$ 
\int_T^{2T} |{\mathcal N}(\tfrac 12+it,k)|^{\frac{2}{k}} |{\mathcal N}(\tfrac 12+it,k-1)|^2 dt \le C_3T(\log T)^{k^2},  
$$ 
for some positive constant $C_3$.
\end{proposition} 

Two applications of H{\" o}lder's inequality 
give 
\begin{align*}
\Big| \int_T^{2T} \zeta(\tfrac 12+it) &{\mathcal N}(\tfrac 12+it, k-1) {\mathcal N}(\tfrac 12-it, k) dt \Big| \\
&\le \Big( \int_T^{2T} |\zeta(\tfrac 12+it)|^{2k} dt \Big)^{\frac 12} \times \Big( \int_{T}^{2T} |\zeta(\tfrac 12+it) {\mathcal N}(\tfrac 12+it,k-1)|^2 dt \Big)^{\frac{1-k}{2}}\\
&\hskip 1 in \times \Big( \int_T^{2T}  |{\mathcal N}(\tfrac 12+it,k)|^{\frac{2}{k}} |{\mathcal N}(\tfrac 12+it,k-1)|^2 dt \Big)^{\frac k2}, 
\end{align*}
so that the lower bound of the theorem follows at once from the three propositions.

 \begin{proof} [Deducing Corollary 2 from Corollary 1]  Let $V$ be in the range of the corollary, and put $k=V/\log \log T$ and 
 $\delta= \log_3 T/\sqrt{\log\log T}$ so that $k+2\delta \le 2$.  The upper bound implicit in the corollary follows (in a stronger form) upon noting that 
 $$ 
 \text{meas} \{ t\in [T, 2T]: |\zeta(\tfrac 12+it)| \ge e^V\} \le e^{-2kV} \int_T^{2T} |\zeta(\tfrac 12+it)|^{2k} dt 
 \ll T \exp\Big( -\frac{V^2}{\log \log T}\Big). 
 $$ 
 To prove the lower bound, consider 
 \begin{equation} 
 \label{1.3}
 \int_T^{2T} |\zeta(\tfrac 12+it)|^{2(k+\delta)} dt \gg (k+\delta) T (\log T)^{(k+\delta)^2} \gg \frac{T}{\sqrt{\log \log T}} (\log T)^{(k+\delta)^2}.  
 \end{equation} 
 The contribution to the integral from $t$ with $|\zeta(\tfrac 12+it)| \le e^V$ is 
 $$ 
 \le e^{2\delta V} \int_{T}^{2T} |\zeta(\tfrac 12+it)|^{2k} dt \ll T (\log T)^{k^2 +2k\delta} = o\Big( \frac{T}{\sqrt{\log \log T}} 
 (\log T)^{(k+\delta)^2}\Big). 
 $$ 
 Similarly, the contribution to the integral from $t$ with $|\zeta(\tfrac 12+it)| \ge e^{V}(\log T)^{2\delta}$ is 
 $$ 
 \le (\log T)^{-2\delta (k+2\delta)}  \int_T^{2T} |\zeta(\tfrac 12+it)|^{2(k+2\delta)} dt \ll T(\log T)^{k^2+ 2\delta k} = o\Big( \frac{T}{\sqrt{\log \log T}} 
 (\log T)^{(k+\delta)^2}\Big). 
 $$ 
 Thus the left side of \eqref{1.3} is dominated by values of $|\zeta(\tfrac 12+it)|$ lying between $e^V =(\log T)^{k}$ and 
 $(\log T)^{k+2\delta}$ and it follows that the measure of the set of such $t$ is 
 \begin{align*}
& \gg (\log T)^{-2(k+\delta) (k+2\delta)} \int_{\substack{ t\in [T,2T] \\ (\log T)^{k+2\delta} \ge |\zeta(\tfrac 12+it)| \ge (\log T)^k} }
 |\zeta( \tfrac 12+it)|^{2(k+\delta)} dt \\
 &\gg \frac{T}{\sqrt{\log \log T}} (\log T)^{-k^2-4k\delta -3\delta^2}.
 \end{align*} 
 The corollary follows.
 \end{proof}


\section{Proof of Proposition \ref{prop1}}

\noindent Expanding out, we have 
\begin{align} 
\label{2.1} 
\int_T^{2T} \zeta(\tfrac 12+it) {\mathcal N}(\tfrac 12+it, k-1) &{\mathcal N}(\tfrac 12-it, k) dt \nonumber \\ 
&= \sum_{n, n\in {\mathcal N}} \frac{(k-1)^{\Omega(n)} k^{\Omega(m)}g(n)g(m)}{\sqrt{mn}} \int_T^{2T} 
\zeta(\tfrac 12+it) \Big(\frac mn\Big)^{it} dt. 
\end{align} 

Using the simple approximation 
\[
\zeta(1/2+it)=\sum_{r \leqs T}\frac{1}{r^{1/2+it}}+O(T^{-1/2}),\qquad t\in[T,2T] 
\] 
we find that 
$$ 
\int_T^{2T} 
\zeta(\tfrac 12+it) \Big(\frac mn\Big)^{it} dt = T \frac{\delta(r n=m)}{\sqrt{r}} + O\Big( T^{\frac12 } 
+ \sum_{\substack{ r \le T \\ r n\neq m}} \frac{1}{\sqrt{r} |\log (r n/m)|} \Big).
 $$
 Here $\delta(r n=m)$ equals $1$ if $n|m$ and $r = m/n$, and there is no main term if $n\nmid m$.  If $r n\neq m$, we may   estimate $1/|\log (r n/m)|$ trivially by $\ll m$, and so the remainder term above is 
 $O(mT^{\frac 12})$.  From these remarks, it follows that the right side of \eqref{2.1} equals 
 \begin{equation} 
 \label{2.2} 
 T \sum_{ \substack{ m, n \in {\mathcal N} \\ n|m}}  \frac{(k-1)^{\Omega(n)} k^{\Omega(m)}g(n)g(m)}{ m} 
 + O\Big( \sum_{ m, n\in {\mathcal N}} \frac{1}{\sqrt{mn}} m T^{\frac 12}\Big). 
 \end{equation} 
Since the elements of ${\mathcal N}$ are all bounded by $T^{1/9}$, the error term above is seen to be $O(T^{7/9})$, which is 
negligible.

Now consider the main term in \eqref{2.2}.   Factor $n=n_2 \cdots n_\ell$ and $m=m_2 \cdots m_\ell$ where $m_j$ and $n_j$ are divisible only by the primes in the interval $(T_{j-1}, T_j)$ and $\Omega(m_j)$ and $\Omega(n_j)$ are bounded by $K_j$. 
Then the main term in \eqref{2.2} factors naturally as 
\begin{equation} 
\label{2.3} 
T \prod_{j=2}^{\ell} \Big( \sum_{ \substack{n_j, m_j \\ n_j | m_j \\ \Omega(m_j) \le 500P_j}} \frac{(k-1)^{\Omega(n_j) }k^{\Omega(m_j)} g(n_j) g(m_j)}{m_j} \Big). 
\end{equation} 
If we drop the condition that $\Omega(m_j) \le K_j$, then the sums over $n_j$, $m_j$ above may be replaced with  (thinking of $a$ as the power of $p$ dividing $m_j$ and $b$ the power dividing $n_j$) 
$$ 
\prod_{T_{j-1} \le p \le T_j} \Big( 1+ \sum_{\substack{ a \ge 1 \\ a \ge b \ge 0}} \frac{k^a (k-1)^b}{p^a} g(p^a) g(p^b) \Big)  
\ge \prod_{T_{j-1} \le p \le T_j} \Big( 1+ \frac{k^2}{p}\Big).
$$  
The error incurred in dropping this condition is bounded in magnitude by 
\begin{align*}
\sum_{\substack{ n_j, m_j \\ n_j|m_j \\ \Omega(m_j) > K_j}} \frac{g(n_j)g(m_j)}{m_j} 
&\le e^{-K_j} \sum_{\substack{ n_j, m_j \\ n_j|m_j }} \frac{g(n_j)g(m_j)}{m_j} e^{\Omega(m_j)} \\
&= e^{-500P_j} 
\prod_{T_{j-1} \le p\le T_j} \Big(1 + \sum_{a\ge 1} \frac{e^a}{a! p^a} \sum_{a\ge b\ge 0} \frac{1}{b!} \Big)\\
&\le e^{-500P_j} \prod_{T_{j-1} \le p \le T_j} \Big( 1+ \frac{20}{p} \Big) \le e^{-400P_j}. 
\end{align*}
It follows that the main term \eqref{2.3} is 
$$ 
\ge T \prod_{j=2}^{\ell} \prod_{T_{j-1} \le p \le T_j} \Big( 1+ \frac{k^2}{p} \Big) \Big(1 - e^{-400P_j}\Big) 
\ge CT (\log T_{\ell})^{k^2}, 
$$ 
for an absolute positive constant $C$.  Since $\log T_\ell \gg k \log T$, and $k^{k^2} \gg 1$ for $0<k \le 1$,  this proves Proposition \ref{prop1}.


\section{Proof of Proposition \ref{prop2}} 

\noindent It is a simple matter to compute the mean square of the zeta function multiplied by a short Dirichlet polynomial.  For example, from \cite{BCH}, we obtain 
\begin{align} 
\label{3.1} 
\int_T^{2T} |\zeta(\tfrac 12+it) {\mathcal N}(\tfrac 12+it, k-1)|^2 dt &= T \sum_{m, n \in {\mathcal N}} \frac{(k-1)^{\Omega(m)+\Omega(n)} g(m) g(n)}{[m,n]} \log \Big( \frac{B T (m,n)^2}{mn} \Big) \nonumber \\ 
&\hskip 1 in +o(T), 
\end{align} 
for a constant $B$.  We must now bound the main term above.   While one can work out an asymptotic for this main term, we give a quick proof of an upper bound, which is all that is needed in Proposition \ref{prop2}.   

Write   
$$ 
\log \Big(  \frac{B T (m,n)^2}{mn} \Big) = \frac{1}{2\pi i } \int_{|z|= 1/\log T} \Big( \frac{BT (m,n)^2}{mn}\Big)^z \frac{dz}{z^2}, 
$$
so that the main term in \eqref{3.1} becomes
$$ 
\frac{T}{2\pi i } \int_{|z|= 1/\log T}\sum_{m, n\in {\mathcal N}} \frac{(k-1)^{\Omega(m)+\Omega(n)} g(m)g(n)}{[m,n]}  \Big( \frac{BT(m,n)^2}{mn}\Big)^z  \frac{dz}{z^2}. 
$$ 
By the triangle inequality, we may estimate the above by 
\begin{equation} 
\label{3.2} 
\le 3T  \log T \max_{|z| = 1/\log T} \Big| \sum_{m, n\in {\mathcal N}} \frac{(k-1)^{\Omega(m)+\Omega(n)} g(m)g(n)}{[m,n]}   
\Big( \frac{(m,n)^2}{mn}\Big)^z\Big|.
\end{equation} 

We can now analyze the sum over $m$ and $n$ in \eqref{3.2} by 
 adapting the argument of the previous section.   Thus decompose $m=m_2 \cdots m_\ell$ and $n=n_2 \cdots n_\ell$ where $m_j$ and $n_j$ are composed only of the primes in $(T_{j-1},T_j)$ and $\Omega(m_j)$ and $\Omega(n_j)$ are both $\le K_j$.  By multiplicativity, the  sum in \eqref{3.2} factors as 
\begin{equation} \label{3.4} 
 \prod_{j=2}^{\ell} \Big( \sum_{\substack { m_j , n_j \\ \Omega(m_j), \Omega(n_j) \le K_j} } \frac{(k-1)^{\Omega(m_j)+\Omega(n_j)} g(m_j) g(n_j)}{[m_j,n_j]}  \Big( \frac{(m,n)^2}{mn} \Big)^z\Big).
\end{equation} 
As before, we handle these terms by first dropping the condition on $\Omega(m_j)$ and $\Omega(n_j)$, and then bounding the error in doing so.   If we drop the conditions on $\Omega(m_j)$ and $\Omega(n_j)$ the sums over $m_j$ and $n_j$ become 
\begin{align*} 
\prod_{T_{j-1} \le p \le T_j} \Big( \sum_{a,b=0}^{\infty} \frac{(k-1)^{a+b}}{a! b! p^{\max(a,b)}} p^{-|b-a|z}\Big) 
&= \prod_{T_{j-1} \le p \le T_j}  \Big( 1+ \frac{(k-1)^2 + 2 (k-1)p^{-z}}{p} + O\Big( \frac{1}{p^2}\Big)\Big)\\
&= \prod_{T_{j-1}  \le p \le T_{j} } \Big( 1 +\frac{k^2-1}{p} + O\Big( \frac{\log p}{p \log T} + \frac{1}{p^2}\Big) \Big).
\end{align*}
The error incurred in dropping the conditions on $\Omega(m_j)$ and $\Omega(n_j)$ is bounded in magnitude by 
\begin{align*}
&\le e^{-K_j} \sum_{m_j, n_j} \frac{g(m)g(n)}{[m,n]} e^{\Omega(m_j)+\Omega(n_j)} \Big(\frac{mn}{(m,n)^2}\Big)^{1/\log T}
\le 
e^{-K_j} \prod_{T_{j-1} \le p\le T_j} \Big(1 + 2\sum_{a=1}^{\infty} \sum_{0\le b\le a} \frac{e^{a+b}}{a! b! p^a} p^{a/\log T} \Big) \\
&\le e^{-500P_j} \prod_{T_{j-1} \le p\le T_j} \Big( 1 + 35 \sum_{a=1}^{\infty} \frac{e^a }{a!p^a} \Big) 
\le e^{-500 P_j} \exp\Big(\sum_{T_{j-1} \le p\le T_j}  \frac{35e}{p} \Big) \le e^{-400P_j}. 
\end{align*}
We conclude that the sum over $m_j$, $n_j$ in \eqref{3.4} is  
\begin{equation} 
\label{3.5} 
\prod_{T_{j-1} \le p\le T_j} \Big(1 + \frac{k^2-1}{p} +O\Big( \frac{\log  p}{p\log  T}  + \frac{1}{p^2}\Big) \Big) \Big(1 + O(e^{-300P_j})\Big), 
\end{equation} 
so that the quantity in \eqref{3.2} is 
$$
 \ll T \log T \prod_{p\le T_\ell} \Big(1 + \frac{k^2-1}{p} + O\Big(\frac{\log p}{p\log  T} + \frac 1{p^2}\Big)  \Big) \ll k^{-1} T (\log T)^{k^2}. 
$$
 The proposition follows. 


\section{Proof of Proposition \ref{prop3}} 

\noindent Recall from \eqref{1.1} and \eqref{2} the definitions of ${\mathcal N}_j(s,\alpha)$ and ${\mathcal N}(s,\alpha)$.  
The following simple lemma is the key to establishing Proposition \ref{prop3}. 

\begin{lemma} \label{lem1} For $2\le j\le \ell$ 
$$ 
|{\mathcal N}_j(\tfrac 12+it, k-1) {\mathcal N}_j(\tfrac 12 +it, k)^{\frac{1}{k}}|^2 \le 
|{\mathcal N}_j(\tfrac 12, +it, k)|^2 (1+ O(e^{-K_j}/k)) + O\Big( 2^{2/k} {\mathcal Q}_j(t)\Big), 
$$ 
where the implied constants are absolute, and 
$$ 
{\mathcal Q}_j(t) =\Big( \frac{12 |{\mathcal P}_j(\tfrac 12+it)|}{K_j} \Big)^{2K_j} \sum_{r=0}^{K_j/k} \Big( \frac{2e |{\mathcal P}_j(\tfrac 12+it)|}{r+1} \Big)^{2r}.
$$ 
\end{lemma} 
\begin{proof}  We begin by observing that if $|z| \le K/10$ then 
$$ 
\Big| \sum_{r=0}^K \frac{z^r}{r!} - e^z \Big| \le \frac{|z|^{K}}{K!} \le \Big(\frac{e}{10}\Big)^{K}, 
$$ 
so that 
\begin{equation} 
\label{4.1} 
\sum_{r=0}^{K} \frac{z^r}{r!} = e^z \Big(1 + O(e^{-K})\Big). 
\end{equation} 
Consider first the case $|{\mathcal P}_j(\tfrac 12+it)| \le K_j/10$, where three applications of \eqref{4.1} show that 
\begin{align*}
|{\mathcal N}_j(\tfrac 12+it, k-1)|^2 | {\mathcal N}_j(\tfrac 12+it,k)|^{\frac{2}{k}} 
&= \exp( 2k \text{Re} {\mathcal P}_j(\tfrac 12+it))\Big( 1+ O(e^{-K_j}/k) \Big)\\
& = |{\mathcal N}_j(\tfrac 12+it,k)|^2 \Big( 1+ O(e^{-K_j}/k) \Big).
\end{align*} 
The lemma follows in this case.  

Suppose now that $|{\mathcal P}_j(\tfrac 12+it)| \ge K_j/10$.  Here note that  
\begin{align} 
\label{4.2} 
|{\mathcal N}_j(\tfrac 12+it, k-1)| &\le \sum_{r=0}^{K_j} \frac{|{\mathcal P}_j(\tfrac 12+it)|^r}{r!} \le 
|{\mathcal P}_j(\tfrac 12+it)|^{K_j} \sum_{r=0}^{K_j} \Big( \frac{10}{K_j}\Big)^{K_j-r} \frac{1}{r!}    \nonumber \\ 
&\le \Big( \frac{12 |{\mathcal P}_j(\tfrac 12+it)|}{K_j}\Big)^{K_j}. 
\end{align} 
Further, applying H{\" o}lder's inequality we find
\begin{align*}
|{\mathcal N}_j(\tfrac 12+it, k)|^{\frac 2k} &\le \Big( \sum_{r=0}^{K_j} \frac{(k |{\mathcal P}_j(\tfrac 12+it)|)^{r}}{r!}\Big)^{\frac 2k} 
\le \Big( \sum_{r=0}^{K_j} \frac{(2k |{\mathcal P}_j(\tfrac 12+it)|)^{\frac{2r}k}}{r!^{2/k}}\Big) \Big( \sum_{r=0}^{K_j} 2^{-r} \Big)^{\frac 2k-1} 
\\ 
&\le 2^{\frac 2k} \sum_{r=0}^{K_j} (2k|{\mathcal P}_j(\tfrac 12+it)|)^{\frac{2r} k} \Big( \frac{e}{r+1} \Big)^{\frac{2r}k}  
\le 2^{\frac 2k} \sum_{r=0}^{K_j}   \Big(\frac{2e|{\mathcal P}_j(\tfrac 12+it)|}{r/k+1}\Big)^{\frac{2r}k}.
\end{align*}
A little calculus allows us to bound the above by 
$$ 
\ll 2^{\frac 2k} \sum_{r=0}^{K_j/k} \Big( \frac{2e |{\mathcal P}_j(\tfrac 12+it)|}{r+1}\Big)^{2r}, 
$$
which when combined with \eqref{4.2} yields the lemma. 
\end{proof}  

We next show that ${\mathcal Q}_j(t)$ (which is always non-negative by definition) is small on average. 

\begin{lemma} 
\label{lem2} With the above notation
$$ 
\int_T^{2T} {\mathcal Q}_j(t) dt \ll T e^{-K_j}. 
$$ 
\end{lemma} 
\begin{proof}  We begin by recalling a simple mean-value theorem for Dirichlet polynomials: 
\begin{align*}
\int_T^{2T} \Big| \sum_{n\le N} a(n) n^{-it}\Big|^2 dt 
&= T \sum_{n\le N} |a(n)|^2 + O\Big( \sum_{\substack{m \neq n \le N}} \frac{|a(m)a(n)|}{|\log (m/n)|}\Big),  
\end{align*}
and bounding $|a(m)a(n)|$ by $|a(m)|^2 +|a(n)|^2$, it follows that 
\begin{equation} 
\label{4.3} 
\int_T^{2T} \Big| \sum_{n\le N} a(n) n^{-it}\Big|^2 dt  = (T+ O(N\log N)) \sum_{n\le N} |a(n)|^2. 
\end{equation} 

Now, for $0\le r\le K_j/k$, 
$$ 
{\mathcal P}_j(\tfrac 12+it)^{K_j+r} = \sum_{ \substack{ \Omega(n) = K_j+r \\ p|n \implies T_{j-1} \le p\le T_j}} \frac{(K_j+r)! g(n)}{n^{\frac 12+it}},
$$ 
is a short Dirichlet polynomial (since $T_j^{K_j(1+1/k)} \le T^{1/10}$), and so by \eqref{4.3} 
\begin{align*} 
\int_T^{2T} |{\mathcal P}_j(\tfrac 12+it)|^{2(K_j+r)} dt &= (T+O(T^{1/2})) \sum_{ \substack{ \Omega(n) = K_j+r \\ p|n \implies T_{j-1} \le p\le T_j}}  \frac{(K_j+r)!^2 g(n)^2}{n} \nonumber \\
&\le (K_j+r)! P_j^{K_j+r} (T+ O(T^{1/2})),
\end{align*} 
where the last bound follows upon noting that $g(n)^2 \le g(n)$.  Using this bound in the definition of ${\mathcal Q}_j(t)$, we find 
\begin{equation} 
\label{4.4}
\int_T^{2T} {\mathcal Q}_j(t) dt \ll T \Big(\frac{12}{K_j}\Big)^{2K_j} \sum_{r=0}^{K_j/k} \Big( \frac{2e}{r+1} \Big)^{2r} (K_j+r)!P_j^{K_j+r} . 
\end{equation} 
Stirling's formula and a little calculus shows that the terms above attain a maximum for $r$ around the solution to 
$r^2=4P_j(K_j+r)$, and since $K_j = 500 P_j$, such $r$ satisfies $2 \sqrt{P_j K_j} \le r \le 2.1 \sqrt{P_jK_j}$.  It follows that the right side of \eqref{4.4} is 
$$ 
\ll T \Big(\frac{12}{K_j}\Big)^{2K_j} \Big( \frac{K_j}{k}\Big) \Big( \frac{2P_j K_j}{e}\Big)^{K_j} e^{2.1\sqrt{P_jK_j}} 
\ll T e^{- K_j}. 
$$ 
\end{proof}

We need one more observation for the proof of the proposition.  Suppose we are given $R$ Dirichlet polynomials 
$$ 
A_j(s) = \sum_{n\in {\mathcal S}_j} a_j(n) n^{-s}, 
$$ 
where the sets ${\mathcal S}_j$ satisfy the following two properties: (i) If $j_1 \neq j_2$ then the elements of ${\mathcal S}_{j_1}$ are all coprime to the elements of ${\mathcal S}_{j_2}$, and (ii) $\prod_{j=1}^{R} n_j \le N$ for all $n_j \in {\mathcal S}_j$.   The coprimality condition implies that there is at most one way to write $n= \prod_{j=1}^{R} n_j$ with $n_j \in {\mathcal S}_j$.   Thus  
applications of \eqref{4.3} give 
\begin{align} 
\label{4.6}
\frac{1}{T} \int_T^{2T} \prod_{j=1}^{R} |A_j(it)|^2 dt 
&= (1+O(NT^{-1}\log N)) \sum_{n\le N} \Big| \sum_{\substack{ n= n_1 \cdots n_R \\ n_j\in {\mathcal S}_j} }\prod_{j=1}^{R} a_j(n_j) \Big|^2  \nonumber
\\
&= (1+O(NT^{-1}\log N)) \prod_{j=1}^R \Big( \sum_{n_j \in {\mathcal S}_j} |a_j(n_j)|^2 \Big) \nonumber \\ 
&= (1+ O(NT^{-1} \log N)) \prod_{j=1}^{R} \Big( \frac 1T \int_{T}^{2T} |A_j(it)|^2 dt \Big). 
\end{align}

We are now ready to combine the above observations to prove Proposition 3.  Applying Lemma \ref{lem1} we find 
\begin{align*}
\int_{T}^{2T} |{\mathcal N}(\tfrac 12+it, k-1)|^2 &|{\mathcal N}(\tfrac 12+it,k)|^{\frac 2k} dt 
\\
&\le \int_{T}^{2T} 
\prod_{j=2}^{\ell} \Big( |{\mathcal N}_j(\tfrac 12+it, k)|^2 (1+O(e^{-K_j}/k)) + O(2^{2/k} {\mathcal Q}_j(t)) \Big) dt.
\end{align*}
Appealing now to the observation \eqref{4.6}, the above is 
\begin{equation} 
\label{4.7}
\ll T \prod_{j=2}^{\ell} \Big( \frac 1T \int_T^{2T}  \Big( |{\mathcal N}_j(\tfrac 12+it, k)|^2 (1+O(e^{-K_j}/k)) + O(2^{2/k} {\mathcal Q}_j(t)) \Big) dt.
\end{equation} 
Applying the mean-value theorem for Dirichlet polynomials \eqref{4.3}, we see that 
\begin{align*}
 \int_T^{2T} |{\mathcal N}_j(\tfrac 12+it, k)|^2 dt &= (T+ O(T^{1/2})) 
 \sum_{\substack{ p|n \implies T_{j-1}\le p\le T_j \\ \Omega(n) \le K_j}}\frac{k^{2\Omega(n)} g(n)^2}{n} 
\\
& \le (T+O(T^{1/2})) \prod_{T_{j-1} \le p\le T_j} \Big(1 +\frac{k^2}{p} + O\Big(\frac{1}{p^2}\Big)\Big). 
\end{align*} 
 Combining this with Lemma \ref{lem2}, we conclude that the quantity in \eqref{4.7} is 
 $$ 
 \ll T \prod_{p\le T_{\ell} } \Big(1+ \frac{k^2}{p} + O\Big( \frac{1}{p^2}\Big) \Big), 
 $$ 
 which completes the proof of the proposition. 
 
 \section{Extensions of the result}  

\noindent We first give the modifications needed to obtain Theorem 1 in the range $k\ge 1$.   Once again let $\ell$ be the 
largest integer with $\log_\ell T \ge 10^4$, and now define $T_j$ by $T_1= k^4 e^2$ and for $2\le j\le \ell$ by 
$$ 
T_j = \exp\Big( \frac{\log T}{k^2 (\log_j T)^2}\Big).  
$$ 
Define ${\mathcal P}_j(s)$, $P_j$ exactly as before, and now put $K_j = 500 k^2P_j$ with ${\mathcal N}(s,\alpha)$ defined accordingly.   Analogously to Proposition 1, we may establish that 
$$ 
\int_T^{2T} \zeta(\tfrac 12+it)  {\mathcal N}(\tfrac 12+it, k-1) {\mathcal N}(\tfrac 12-it,k) dt 
\gg  T \prod_{T_1 \le p \le T_\ell} \Big( 1+  \frac{k^2}{p} \Big). 
$$ 
Now H{\" o}lder's inequality gives that the left side above is 
$$ 
\le \Big( \int_T^{2T} |\zeta(\tfrac 12+it)|^{2k} dt \Big)^{\frac{1}{2k}} \Big( \int_T^{2T} |{\mathcal N}(\tfrac 12+it, k-1) {\mathcal N}(\tfrac 12+it, k)|^{\frac{2k}{2k-1}} dt \Big)^{\frac{2k-1}{2k}}. 
$$ 
By modifying the argument of Proposition 3 (indeed the details are even a little simpler) the second term above may be bounded 
by 
$$
\ll \Big( T \prod_{T_1 \le p \le T_\ell} \Big( 1+ \frac{k^2}{p} + O\Big( \frac{k^4}{p^2} \Big) \Big)^{\frac{2k-1}{2k}}  \ll 
\Big(  T  \prod_{T_1\le p\le T_\ell} \Big( 1+\frac{k^2}{p} \Big) \Big)^{\frac{2k-1}{2k}}.
$$ 
  The lower bound claimed in the theorem follows.

Examining our proof, we may extract the following principle.  Given a family of $L$-functions, if one can compute the mean value of $L(\tfrac 12)$ multiplied by suitable short Dirichlet polynomials, as well as the mean value of $|L(\tfrac 12)|^2$ multiplied by suitable short Dirichlet polynomials, then one obtains a lower bound of the right order for the moments $|L(\tfrac 12)|^k$ for all $k>0$.  If $k \ge 1$, then one needs only an understanding of the mean value of $L(\tfrac 12)$ multiplied by short Dirichlet polynomials, and knowledge of the second moment of $L(\tfrac 12)$ is not required.  Thus, for example, one may establish that 
\begin{equation} \label{6.1} 
\sum_{\chi \pmod{q}} |L(\tfrac 12,\chi)|^{2k} \gg_k q (\log q)^{k^2}, 
\end{equation} 
where $q$ is a large prime, and $k>0$.  Or, that for $k>0$ and large $X$ 
\begin{equation} \label{6.2} 
\sum_{|d|\le X}^{\flat} |L(\tfrac 12,\chi_d)|^k \gg_k X (\log X)^{\frac{k(k+1)}{2}}, 
\end{equation} 
where the sum is over fundamental discriminants $d$.  Previously, \eqref{6.1} and \eqref{6.2} were accessible for all $k\ge 1$ by 
\cite{RS lower}, and \eqref{6.1} was known for rational $0 \le k\le 1$ by the work of Chandee and Li \cite{CL}.    A third example is the family  of quadratic twists of a newform $f$, where the second moment of the central $L$-values is not known.  Here one can establish 
\begin{equation} 
\label{6.3} 
\sum_{|d|\le X}^{\flat} L(\tfrac 12, f\times \chi_d)^k \gg_k X(\log X)^{\frac{k(k-1)}{2}},
\end{equation} 
for all $k\ge 1$.  Such a result would be accessible also to the method of \cite{RS lower}, but the problem of obtaining satisfactory lower bounds for the small moments $k<1$ (which is connected to the delicate question of non-vanishing of $L$-values) remains open.

\end{document}